\documentclass[10pt,a4paper]{article}
\usepackage[margin=1.1cm]{geometry}
\usepackage{amsmath,amsthm,amsfonts,amssymb,amscd,cite,graphicx}

\usepackage{titlesec}
\titleformat{\section}
{\normalfont\fontsize{12}{15}\bfseries}{\thesection}{1em.}{}

\newtheorem{proposition}{Proposition}[section]

\newtheorem{lemma}{Lemma}[section]

\newtheorem{theorem}{Theorem}[section]

\allowdisplaybreaks[4]

\let\oldbibliography\thebibliography
\renewcommand{\thebibliography}[1]{%
  \oldbibliography{#1}%
  \setlength{\itemsep}{-2pt}%
}

\newcommand{\RR}{\mathbb{R}}

\renewcommand{\a}{\alpha}
\renewcommand{\b}{\beta}
\renewcommand{\d}{\delta}
\newcommand{\D}{\Delta}

\begin{document}

\baselineskip=0.20in

\makebox[\textwidth]{%
\hglue-15pt
\begin{minipage}{0.6cm}	
\vskip9pt
\end{minipage} \vspace{-\parskip}
\begin{minipage}[t]{6cm}
\footnotesize{ {\bf Discrete Mathematics Letters} \\ \underline{www.dmlett.com}}
\end{minipage}
\hfill
\begin{minipage}[t]{6.5cm}
\normalsize {\it Discrete Math. Lett.}  {\bf X} (202X) XX--XX
\end{minipage}}
\vskip36pt

\noindent
{\large \bf Mean Sombor index}\\

\noindent
J. A. M\'endez-Berm\'udez$^{1,}\footnote{Corresponding author (jmendezb@ifuap.buap.mx)}$, R. Aguilar-S\'anchez$^{2}$,  Edil D. Molina$^{3}$, Jos\'e M. Rodr\'{\i}guez$^{4}$\\

\noindent
\footnotesize
$^1${\it Instituto de F\'{\i}sica, Benem\'erita Universidad Aut\'onoma de Puebla, Apartado Postal J-48, Puebla 72570, Mexico} \\
\noindent
$^2${\it Facultad de Ciencias Qu\'imicas, Benem\'erita Universidad Aut\'onoma de Puebla,
Puebla 72570, Mexico}\\
\noindent
$^3${\it Facultad de Matem\'aticas, Universidad Aut\'onoma de Guerrero, Carlos E. Adame No.54 Col. Garita, Acapulco Gro. 39650, Mexico} \\
\noindent
$^4${\it Departamento de Matem\'aticas, Universidad Carlos III de Madrid, Avenida de la Universidad 30,
28911 Legan\'es, Madrid, Spain} \\

\noindent
 (\footnotesize Received: Day Month 202X. Received in revised form: Day Month 202X. Accepted: Day Month 202X. Published online: Day Month 202X.)\\

\setcounter{page}{1} \thispagestyle{empty}

\baselineskip=0.20in

\normalsize

 \begin{abstract}
 \noindent
  We introduce a degree--based variable topological index inspired on the power (or generalized) mean. We name this new index as the mean Sombor index: $mSO_\alpha(G) = \sum_{uv \in E(G)} \left[\left( d_u^\alpha+d_v^\alpha \right) /2 \right]^{1/\alpha}$. Here, $uv$ denotes the edge of the graph $G$ connecting the vertices $u$ and $v$, $d_u$ is the degree of the vertex $u$, and $\alpha \in \mathbb{R} \backslash \{0\}$. We also consider the limit cases $mSO_{\alpha\to 0}(G)$ and $mSO_{\alpha\to\pm\infty}(G)$. Indeed, for given values of $\alpha$, the mean Sombor index is related to well-known topological indices such as the inverse sum indeg index, the reciprocal Randic index, the first Zagreb index, the Stolarsky--Puebla index and several Sombor indices. Moreover, through a quantitative structure property relationship (QSPR) analysis we show that $mSO_\alpha(G)$ correlates well with several physicochemical properties of octane isomers. Some mathematical properties of mean Sombor indices as well as bounds and new relationships with known topological indices are also discussed.  \\[2mm]
 \end{abstract}

\baselineskip=0.20in

\section{Preliminaries}

For two positive real numbers $x,y$ the power mean or generalized mean
$PM_\alpha(x,y)$ with exponent $\alpha\in \mathbb{R} \backslash \{0\}$ is given as
\begin{equation}
	PM_\alpha(x,y)=\left( \frac{x^\alpha+y^\alpha}{2} \right)^{1/\alpha} \, ,
	\label{Mxy}
\end{equation}
see e.g.~\cite{B03,S09}. $PM_\alpha(x,y)$ is also known as H\"older mean.
For given values of $\alpha$, $PM_\alpha(x,y)$ reproduces well-known mean values. As examples, in Table~\ref{TableMxy}
we show some expressions for $PM_\alpha(x,y)$ for selected values of $\alpha$ with their corresponding
names, when available.

\begin{table}[b!]
	\caption{Expressions for the generalized mean $PM_\alpha(x,y)$ for selected values of $\alpha$.} 
	\centering 
	\begin{tabular}{r l l} 
		\hline\hline 
		$\alpha$ & $PM_\alpha(x,y)$ & name (when available) \\ [0.5ex] 
		\hline 
		$-\infty$ & $PM_{\alpha\to-\infty}(x,y)=\min(x,y)$ & minimum value  \\ [1ex]
		$-1$ & $\displaystyle PM_{-1}(x,y)=\frac{2xy}{x+y}$ & harmonic mean  \\ [1ex]
		0 & $\displaystyle PM_{\alpha\to 0}(x,y)=\sqrt{xy}$ & geometric mean  \\ [1ex]
		$1/2$ & $\displaystyle PM_{1/2}(x,y)=\left( \frac{\sqrt{x}+\sqrt{y}}{2} \right)^2$ &  \\ [1ex]
		1 & $\displaystyle PM_1(x,y)=\frac{x+y}{2}$ & arithmetic mean  \\ [1ex]
		2 & $\displaystyle PM_2(x,y)=\left( \frac{x^2+y^2}{2} \right)^{1/2}$ & root mean square \\ [1ex]
		3 & $\displaystyle PM_3(x,y)=\left( \frac{x^3+y^3}{2} \right)^{1/3}$ & cubic mean \\ [1ex]
		$\infty$ & $PM_{\alpha\to\infty}(x,y)=\max(x,y)$ & maximum value  \\ [1ex] 
		\hline 
	\end{tabular}
	\label{TableMxy} 
\end{table}

There is a well known inequality for the power mean, namely~\cite{OT57,C66,L74}:
For any $\alpha_1<\alpha_2$,
\begin{equation}
	PM_{\alpha_1}(x,y) \le PM_{\alpha_2}(x,y) \, ,
	\label{ineq1}
\end{equation}
where the equality is attained for $x=y$.

\section{The mean Sombor index}

A large number of graph invariants of the form
\begin{equation}
	TI(G) = \sum_{uv \in E(G)} F(d_u,d_v)
	\label{TI}
\end{equation}
are currently been studied in mathematical chemistry; where $uv$ denotes the edge of the graph $G$ connecting the vertices $u$ and $v$,
$d_u$ is the degree of the vertex $u$, and $F(x,y)$ is an appropriate chosen function, see e.g.~\cite{G13,S15,S21}.

Inspired by the power mean and given a simple graph $G=(V(G),E(G))$, here we choose
the function $F(x,y)$ in Eq.~(\ref{TI}) as the power mean $PM_\alpha(x,y)$ and define the degree--based variable topological index
\begin{equation}
	mSO_\alpha(G)= \sum_{uv \in E(G)} \left( \frac{d_u^\alpha+d_v^\alpha}{2} \right)^{1/\alpha} ,
	\label{MG}
\end{equation}
where $\alpha\in \mathbb{R} \backslash \{0\}$. We name $mSO_\alpha(G)$ as the \emph{mean Sombor index}.

\begin{table}[b!]
	\caption{Expressions for the mean Sombor index $mSO_\alpha(G)$ for selected values of $\alpha$.} 
	\centering 
	\begin{tabular}{r l l} 
		\hline\hline 
		$\alpha$ & $mSO_\alpha(G)$ & index equivalence \\ [0.5ex] 
		\hline 
		$-\infty$ & $\displaystyle mSO_{\alpha\to-\infty}(G)=\sum_{uv \in E(G)} \min(d_u,d_v)$ & $\displaystyle SP_{\alpha\to-\infty}(G)$   \\ [1ex]
		$-1$ & $\displaystyle mSO_{-1}(G)=\sum_{uv \in E(G)} \frac{2d_ud_v}{d_u+d_v}$ & $2ISI(G)$  \\ [1ex]
		0 & $\displaystyle mSO_{\alpha\to 0}(G)=\sum_{uv \in E(G)} \sqrt{d_ud_v}$ & $R^{-1}(G)$  \\ [1ex]
		$1/2$ & $\displaystyle mSO_{1/2}(G)=\sum_{uv \in E(G)} \left( \frac{\sqrt{d_u}+\sqrt{d_v}}{2} \right)^2$ & $2^{-2} KA^1_{1/2,2}(G)$ \\ [1ex]
		1 & $\displaystyle mSO_1(G)=\sum_{uv \in E(G)} \frac{d_u+d_v}{2}$ & $2^{-1}M_1(G)$  \\ [1ex]
		2 & $\displaystyle mSO_2(G)=\sum_{uv \in E(G)} \left( \frac{d_u^2+d_v^2}{2} \right)^{1/2}$ & $2^{-1/2}SO(G)$ \\ [1ex]
		3 & $\displaystyle mSO_3(G)=\sum_{uv \in E(G)} \left( \frac{d_u^3+d_v^3}{2} \right)^{1/3}$ & $2^{-1/3} KA^1_{3,1/3}(G)$ \\ [1ex]
		$\infty$ & $\displaystyle mSO_{\alpha\to\infty}(G)=\sum_{uv \in E(G)} \max(d_u,d_v)$ & $\displaystyle SP_{\alpha\to\infty}(G)$  \\ [1ex] 
		\hline 
	\end{tabular}
	\label{TableMG} 
\end{table}

Note, that for given values of $\alpha$, the mean Sombor index is related to known topological indices:
$mSO_{-1}(G) = 2ISI(G)$, where $ISI(G)$ is the inverse sum indeg index~\cite{VG10,V10},
$mSO_{\alpha\to 0}(G) = R^{-1}(G)$, where $R^{-1}(G)$ is the reciprocal Randic index~\cite{GFE14}, and
$mSO_1(G) = M_1(G)/2$, where $M_1(G)$ is the first Zagreb index~\cite{GT72}.
Also, it is relevant to stress that the mean Sombor index is related to several Sombor indices:
$mSO_2(G) = 2^{-1/2}SO(G)$, where $SO(G)$ is the Sombor index~\cite{G21a},
$mSO_\alpha(G) = 2^{-1/\alpha} SO_\alpha(G)$, where $SO_\alpha(G)$ is the $\alpha$-Sombor index~\cite{RDA21}, and
$mSO_\alpha(G) = 2^{-1/\alpha} KA^1_{\alpha,1/\alpha}(G)$, where
$KA^1_{\alpha,\beta}(G)= \sum_{uv\in E(G)} \left( d_u^\alpha + d_v^\alpha \right)^\beta$
is the first $(\alpha,\beta)-KA$ index~\cite{K19}.
In addition, the limit cases $mSO_{\alpha\to\pm\infty}(G)$ correspond with the limit cases $SP_{\alpha\to\pm\infty}(G)$
of the recently introduced Stolarsky--Puebla index~\cite{MAAS22}.

In Table~\ref{TableMG} we report some expressions for $mSO_\alpha(G)$
for selected values of $\alpha$ that we identify with known topological indices.

\section{QSPR study of $mSO_\alpha(G)$ on octane isomers}
\label{QSPR}

As a first application of mean Sombor indices, here we perform a quantitative structure property
relationship (QSPR) study of $mSO_\alpha(G)$
to model some physicochemical properties of octane isomers.
Here we choose to study the following properties:
acentric factor (AcentFac), boiling point (BP), heat capacity at constant pressure (HCCP),
critical temperature (CT), relative density (DENS), standard enthalpy of formation (DHFORM),
standard enthalpy of vaporization (DHVAP), enthalpy of formation (HFORM), heat of vaporization (HV) at 25$^{\circ}$C,
enthalpy of vaporization (HVAP), and entropy (S).
The experimental values of the physicochemical
properties of the octane isomers were kindly provided by Dr. S. Mondal, see Table 2 in Ref.~\cite{MDDP21}.

\begin{figure}[ht!]
	\begin{center}
		\includegraphics[scale=0.6]{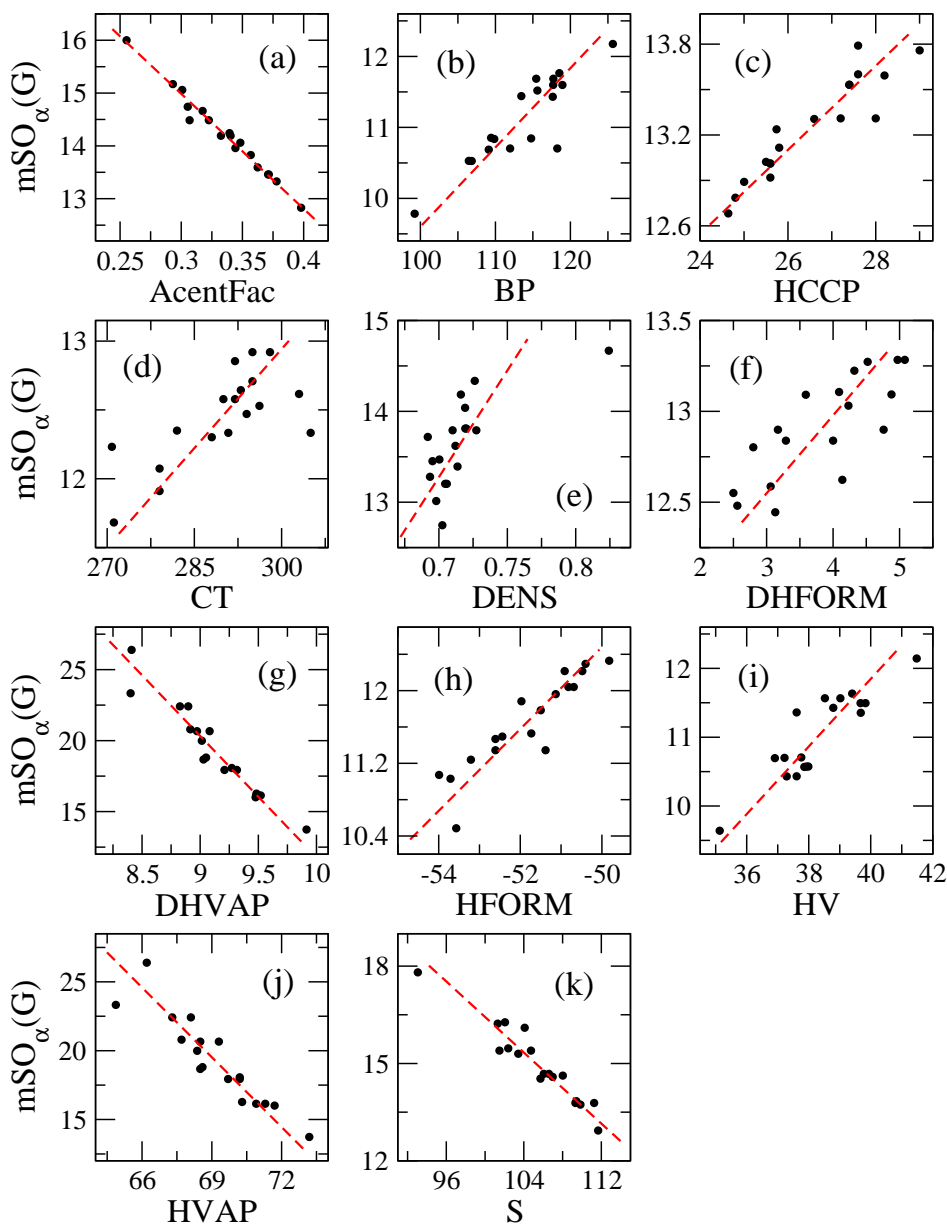}
		\caption{\footnotesize{
		Mean Sombor index $mSO_\alpha(G)$ vs.~the physicochemical properties of octane isomers, for the
		values of $\alpha$ that maximize the correlations: (a) $\alpha=0$, (b) $\alpha=-8.19$, (c)
		$\alpha=-0.87$, (d) $\alpha=-2.05$, (e) $\alpha=-0.53$, (f) $\alpha=-1.28$, (g) $\alpha\rightarrow\infty$,
		(h) $\alpha=-4.23$, (i) $\alpha\rightarrow-\infty$, (j) $\alpha\rightarrow\infty$, and (k) $\alpha=0.58$.
		Red dashed lines are the linear QSPR models of Eq.~(\ref{P}), with the regression and statistical parameters
		resumed in Table~\ref{a_values}.
		}}
		\label{Fig01}
	\end{center}
\end{figure}

In Fig.~\ref{Fig01} we plot $mSO_\alpha(G)$ vs.~the physicochemical properties of octane isomers for the
values of $\alpha$ that maximize the absolute value of Pearson's correlation coefficient $r$; see Table~\ref{a_values}.
Moreover, in Fig.~\ref{Fig01} we tested the following linear
regression model
\begin{equation}
	{\cal P} = c_1 [mSO_\alpha(G)] + c_2,
	\label{P}
\end{equation}
where ${\cal P}$ represents a given physicochemical property.
In Table~\ref{a_values} we resume the regression and statistical parameters of the linear QSPR models
(see the red dashed lines in Fig.~\ref{Fig01}) given by Eq.~(\ref{P}).

\begin{table}[h!]
	\caption{For the physicochemical properties ${\cal P}$ reported in Fig.~\ref{Fig01}: values of $\alpha$ that maximize the absolute value of Pearson's correlation coefficient $r$. $c_2$, $c_1$, $SE$, $F$, and $SF$ are the intercept, slope, standard error, $F$-test, and statistical significance, respectively, of the linear QSPR models of Eq.~(\ref{P}).}
	\label{a_values}
	\centering
	\begin{tabular}{r r r r r r r r }
		\hline\hline
		property ${\cal P}$ & $\a$      & $r$       & $c_2$        & $c_1$       & $SE$     & $F$ & $SF$  \\ [0.5ex]
		\hline
		AcentFac			& $0$	& $-0.990$  & $0.988$   & $-0.046$ & $0.005$ & $749.116$ & $7.25$E-15
		\\ [1ex]
		BP					& $-8.19$	& $0.886$  & $14.115$  & $8.946$  & $2.929$ & $58.126$ & $1.00$E-06
		\\ [1ex]
		HCCP				& $-0.87$	& $0.928$  & $-21.216$ & $3.604$  & $0.504$ & $98.128$ & $3.13$E-08
		\\ [1ex]
		CT					& $-2.05$	& $0.717$  & $30.048$  & $20.859$ & $6.788$ & $16.934$ & $8.10$E-04
		\\ [1ex]
		DENS				& $-0.53$	& $0.702$  & $0.134$   & $0.042$  & $0.022$ & $15.518$ & $1.17$E-03
		\\ [1ex]
		DHFORM				& $-1.28$	& $0.781$  & $-26.253$ & $2.331$  & $0.546$ & $24.924$ & $1.33$E-04
		\\ [1ex]
		DHVAP				& $\infty$	& $-0.962$  & $11.375$  & $-0.117$ & $0.108$ & $196.401$ & $2.11$E-10
		\\ [1ex]
		HFORM				& $-4.23$	& $0.912$  & $-77.705$ & $2.220$  & $0.530$ & $78.903$ & $1.39$E-07
		\\ [1ex]
		HV					& $-\infty$ & $0.895$  & $15.880$  & $2.036$   & $1.286$ & $4.622$ & $4.72$E-02
		\\ [1ex]
		HVAP				& $\infty$	& $-0.921$  & $80.550$  & $-0.592$ & $0.813$ & $89.724$ & $5.81$E-08
		\\ [1ex]
		S					& $0.58$	& $-0.956$  & $160.060$ & $-3.655$ & $1.372$ & $98.128$ & $3.13$E-08
		\\ [1ex]
		\hline

	\end{tabular}
\end{table}

From Table~\ref{a_values} we can conclude that $mSO_\alpha(G)$ provides good predictions of
AcentFac, BP, HCCP, DHVAP, HFORM, HV, HVAP, and S for which the
correlation coefficients (absolute values) are closer or higher than 0.9.
Note that for all these physicochemical properties of octane isomers the statistical significance
of the linear model of Eq.~(\ref{P}) is far below 5\%.
Also notice that the mean Sombor index that better correlates (linearly) with the AcentFac is $mSO_{\alpha\to 0}(G)$,
which indeed coincides with the reciprocal Randic index.
Moreover, we found that $|r|$ is maximized when $\alpha \to \infty$, for DHVAP and HVAP, and when
$\alpha \to -\infty$ for HV. This means that the limiting cases $mSO_{\alpha\to\pm\infty}(G)$ are also
relevant from an application point of view.

\section{Inequalities involving $mSO_\alpha(G)$}
\label{ineq}

Equation~(\ref{ineq1}) can be straightforwardly used to
state a monotonicity property for the $mSO_\alpha(G)$ index, as well as inequalities for related indices.
That is, if $ \alpha_1<\alpha_2 $ we have,
\begin{equation}
	mSO_{\alpha_1}(G) \le mSO_{\alpha_2}(G) \, ,
	\label{ineq11}
\end{equation}
which implies, for the the first $(a,b)-KA$ index, that
\begin{equation}
	2^{-1/\alpha_1} KA^1_{\alpha_1,1/\alpha_1}(G) \le 2^{-1/\alpha_2} KA^1_{\alpha_2,1/\alpha_2}(G) \, , \qquad \alpha_1<\alpha_2 \, ,
	\label{ineq12}
\end{equation}
and moreover
\begin{equation}
	2ISI(G) \le R^{-1}(G) \le 2^{-2} KA^1_{1/2,2}(G) \le 2^{-1}M_1(G) \le 2^{-1/2}SO(G) \, .
	\label{ineq13}
\end{equation}
Note that this last inequality involves the inverse sum indeg index, the reciprocal Randic index, the $(a,b)-KA$ index,
the first Zagreb index, and the Sombor index. It is fair to acknowledge  that
the very last inequality in (\ref{ineq13}) was already included in the Theorem 3.1 of~\cite{MMM21}.

In what follows we will state bounds for the mean Sombor index as well as new relationships with known topological indices.

We will use the following particular case of Jensen's inequality.

\begin{lemma} \label{l:Jensen}
	If $g$ is a convex function on $\RR$ and $x_1,\dots,x_m \in \RR$, then
	$$
	g\Big(\frac{ x_1+\cdots +x_m }{m} \Big) \le \frac{1}{m} \, \big(g( x_1)+\cdots +g(x_m) \big) .
	$$
	If $g$ is strictly convex, then the equality is attained in the inequality if and only if $x_1=\dots=x_m$.
\end{lemma}

\begin{theorem}
	Let $G$ be a graph with $m$ edges and $\a\in \RR $; if $\a>1$ then
	$$
	mSO_\a(G) \le \frac{m^{1-1/\a}}{2^{1/\a}}\left( M_1^{\a+1}(G) \right)^{1/\a} \, ,
	$$
	if $\a<1$ and $\a\ne 0$ then
	$$
	mSO_\a(G) \ge \frac{m^{1-1/\a}}{2^{1/\a}}\left( M_1^{\a+1}(G) \right)^{1/\a} \, ,
	$$
	and the equality in each bound is attained for a connected graph $G$
if and only if $G$ is regular or biregular.
\end{theorem}

\begin{proof}
	Assume first that $\a>1$ then, for $x\ge 0$, $x^{1/\a}$ is a concave function and by Lemma~\ref{l:Jensen} we have
	$$
	\begin{aligned}
		\frac{1}{m}\sum_{uv\in E(G)} \left( \frac{d_u^\a + d_v^\a}{2} \right)^{1/\a}
		&\le
		\left( \frac{1}{2m} \sum_{uv\in E(G)} (d_u^\a + d_v^\a)  \right)^{1/\a} \\
		&=
		\frac{1}{2^{1/\a} m^{1/\a}}\left( \sum_{u\in V(G)} d_u^{\a+1} \right)^{1/\a}
		=
		\frac{1}{2^{1/\a} m^{1/\a}}\left( M_1^{\a+1}(G) \right)^{1/\a}  \, .
	\end{aligned}
	$$

	Assume now that $\a<1$ and $\a\ne 0$, then $x^{1/\a}$ is a convex function and by Jensen's inequality we obtain
	$$
	\begin{aligned}
		\frac{1}{m}\sum_{uv\in E(G)} \left( \frac{d_u^\a + d_v^\a}{2} \right)^{1/\a}
		&\ge
		\left( \frac{1}{2m} \sum_{uv\in E(G)} (d_u^\a + d_v^\a)  \right)^{1/\a} \\
		&=
		\frac{1}{2^{1/\a} m^{1/\a}}\left( \sum_{u\in V(G)} d_u^{\a+1} \right)^{1/\a}
		=
		\frac{1}{2^{1/\a} m^{1/\a}}\left( M_1^{\a+1}(G) \right)^{1/\a} .
	\end{aligned}
	$$
	If $G$ is regular or biregular, with maximum and minimum degrees $\D$ and $\d$, respectively,
	$$
	mSO_\a(G)=m\left( \frac{\D^\a+\d^\a}{2} \right)^{1/\a}
= \frac{m^{1-1/\a}}{2^{1/\a}} \left(m( \D^\a+\d^\a) \right)^{1/\a}
= \frac{m^{1-1/\a}}{2^{1/\a}} \left( M_1^{\a+1}(G) \right)^{1/\a} .
	$$
	If any of these equalities hold, for every $uv,u^\prime v^\prime \in E(G)$, by Lemma~\ref{l:Jensen}, we have $d_u^\a+d_v^\a=d_{u^\prime}^\a + d_{v^\prime}^\a$. In particular if we take $u=u^\prime$ we have $d_v = d_{v^\prime}$, so all the neighbors of a vertex $u\in V(G)$ have the same degree. Thus, since $G$ is a connected graph, $G$ is regular or biregular.
\end{proof}

In order to prove the next result we need an additional technical result.
In~\cite[Theorem 3]{BMRS} appears a converse of H\"older inequality, which in the discrete case can be stated as follows~\cite[Corollary 2]{BMRS}.

\begin{lemma} \label{c:holder}
	If $1<p,q<\infty$ with $1/p+1/q=1$, $x_j,y_j\ge 0$ and $a y_j^q \le x_j^p \le b y_j^q$ for $1\le j \le k$ and some positive constants $a,b,$ then:
	$$
	\Big(\sum_{j=1}^k x_j^p \Big)^{1/p} \Big(\sum_{j=1}^k y_j^q \Big)^{1/q}
	\le K_p(a ,b ) \sum_{j=1}^k x_jy_j ,
	$$
	where
	$$
	K_p(a ,b )
	= \begin{cases}
		\,\displaystyle\frac{1}{p} \Big( \frac{a }{b } \Big)^{1/(2q)} + \frac{1}{q} \Big( \frac{b }{a } \Big)^{1/(2p)}\,, & \quad \text{if } 1<p<2 \, ,
		\\
		\, & \,
		\\
		\,\displaystyle\frac{1}{p} \Big( \frac{b }{a } \Big)^{1/(2q)} + \frac{1}{q} \Big( \frac{a }{b } \Big)^{1/(2p)}\,, & \quad \text{if } p \ge 2 \, .
	\end{cases}
	$$
	
	If $x_j>0$ for some $1\le j \le k$, then the equality in the bound is attained if and only if $a =b$ and $x_j^p=a y_j^q$ for every $1\le j \le k$.
\end{lemma}

\begin{theorem}
	Let $G$ be a graph with $m$ edges, maximum degree $\D$ and minimum degree $\d$, let $0 < \a < 1 $, then
	
	$$
	mSO_\a(G) \le  \frac{m^{1-1/\a}}{2^{1/\a}} \, K_{\a} \left(  M_1^{\a+1}(G) \right)^{1/\a}
	$$
	where
	$$
	K_{\a}^{\a} = \left\lbrace
	\begin{array}{cc}
		\a\left( \frac{\D}{\d}\right)^{\frac{\a - \a^2}{2}} + (1-\a)\left( \frac{\D}{\d}\right)^{\frac{-\a^2}{2}}, & \text{if } 0<\a\le\frac{1}{2} \;,\\
		
		\a\left( \frac{\D}{\d}\right)^{\frac{\a^2-\a}{2}} + (1-\a)\left( \frac{\D}{\d}\right)^{\frac{\a^2}{2}}, & \text{if } \frac{1}{2} < \a < 1 \, ,
	\end{array}
	\right.
	$$
	the equality holds if and only if $G$ is a regular graph.
	
\end{theorem}

\begin{proof}
	For each $uv\in E(G)$ we have
	$$
	\d^\a \le \frac{d_u^\a + d_v^\a}{2} \le \D^\a\, .
	$$
	If we take $x_j=d_u^\a$, $y_j=d_v^\a$ and $p=1/\a$  by Lemma~\ref{c:holder} we have
	$$
	\begin{aligned}
		m^{1-\a}\left( mSO_\a(G) \right) ^\a
		& =
		\left( \sum_{uv\in E(G)} \left( \frac{d_u^\a + d_v^\a}{2} \right)^{1/\a}\right)^{\a}
		\left( \sum_{uv\in E(G)} 1^{\frac{1}{1-\a}} \right)^{1-\a} \\
		& \le K_{\a}^{\a} \sum_{uv\in E(G)}\frac{d_u^{\a}+d_v^{\a}}{2}
		= \frac{1}{2} \, K_{\a}^{\a} M_1^{\a+1}(G) \, ,
	\end{aligned}
	$$	
	where
	$$
	K_{\a}^{\a}= \left\lbrace
	\begin{array}{cc}
		\a\left( \frac{\D}{\d}\right)^{\frac{\a - \a^2}{2}} + (1-\a)\left( \frac{\D}{\d}\right)^{\frac{-\a^2}{2}}, & \text{if } 0<\a\le\frac{1}{2} \;, \\
		
		\a\left( \frac{\D}{\d}\right)^{\frac{\a^2-\a}{2}} + (1-\a)\left( \frac{\D}{\d}\right)^{\frac{\a^2}{2}}, & \text{if } \frac{1}{2} < \a < 1 \, ,
	\end{array}
	\right.
	$$
	and the equality holds if and only if $\d=\D$, i.e., $G$ is regular.
\end{proof}

The following inequalities are known for $x,y > 0$:
\begin{equation}\label{eq:a}
	\begin{aligned}
		x^a + y^a
		< (x + y)^a
		\le 2^{a-1}(x^a + y^a )
		\qquad & \text{if } \, a > 1,
		\\
		2^{a-1}(x^a + y^a )
		\le (x + y)^a
		< x^a + y^a
		\qquad & \text{if } \, 0<a < 1,
		\\
		(x + y)^a
		\le 2^{a-1}(x^a + y^a )
		\qquad & \text{if } \, a < 0,
	\end{aligned}
\end{equation}
and the equality in the second, third or fifth bound is attained for each $a$ if and only if $x=y$.

\begin{proposition}
	Let $G$ be a graph and $\a \in \RR \backslash\{0\}$, then
	$$
	\begin{aligned}
		2^{-1/\a}SO(G)
		< mSO_\a(G)
		\le 2^{-1/2}SO(G)
		\qquad & \text{if } \, 0< \a < 2,
		\\
		2^{-1/2}SO(G)
		\le mSO_\a(G)
		<2^{-1/\a}SO(G)
		\qquad & \text{if } \, \a > 2,
		\\
		mSO_\a(G)
		\le 2^{-1/2}SO(G)
		\qquad & \text{if } \, \a<0 \, ,
	\end{aligned}
	$$
	and the equality in the second, third or fifth bound is attained for each $\a$ if and only if each connected component of $G$ is a regular graph.
\end{proposition}

\begin{proof}
	If we divide each one of the inequalities in (\ref{eq:a}) by $2^a$ we obtain
	$$
	\begin{aligned}
		2^{-a}\left( {x^a + y^a}\right)
		< \left( \frac{x + y}{2}\right) ^a
		\le \frac{x^a + y^a}{2}
		\qquad & \text{if } \, a > 1,
		\\
		\frac{x^a + y^a}{2}
		\le \left( \frac{x + y}{2} \right) ^a
		<2^{-a} \left( {x^a + y^a} \right)
		\qquad & \text{if } \, 0<a < 1,
		\\
		\left( \frac{x + y}{2} \right) ^a
		\le \frac{x^a + y^a}{2}
		\qquad & \text{if } \, a < 0 .
	\end{aligned}
	$$
	If we take $x=d_u^\a$, $y=d_v^\a$ and $a=2/\a$; then the previous inequalities give
	$$
	\begin{aligned}
		2^{-2/\a}\left( {d_u^2 + d_v^2}\right)
		< \left( \frac{d_u^\a + d_v^\a}{2}\right) ^{2/\a}
		\le \frac{d_u^{2} + d_v^{2}}{2}
		\qquad & \text{if } \, 0< \a < 2,
		\\
		\frac{d_u^{2} + d_v^{2}}{2}
		\le \left( \frac{d_u^\a + d_v^\a}{2} \right) ^{2/\a}
		<2^{-{2/\a}} \left( {d_u^{2} + d_v^{2}} \right)
		\qquad & \text{if } \, \a > 2,
		\\
		\left( \frac{d_u^\a + d_v^\a}{2} \right) ^{2/\a}
		\le \frac{d_u^{2} + d_v^{2}}{2}
		\qquad & \text{if } \, \a < 0 ,
	\end{aligned}
	$$
	and the equality in the second, third or fifth bounds are attained for each $a$ if and only if $d_u=d_v$.
	From this we obtain
	$$
	\begin{aligned}
		2^{-1/\a}\left( {d_u^2 + d_v^2}\right) ^{1/2}
		< \left( \frac{d_u^\a + d_v^\a}{2}\right) ^{1/\a}
		\le 2^{-1/2 } \left( {d_u^{2} + d_v^{2}} \right) ^{1/2}
		\qquad & \text{if } \, 0< \a < 2,
		\\
		2^{-1/2 } \left( {d_u^{2} + d_v^{2}} \right) ^{1/2}
		\le \left( \frac{d_u^\a + d_v^\a}{2} \right) ^{1/\a}
		<2^{-1/\a}\left( {d_u^2 + d_v^2}\right) ^{1/2}
		\qquad & \text{if } \, \a > 2,
		\\
		\left( \frac{d_u^\a + d_v^\a}{2} \right) ^{1/\a}
		\le 2^{-1/2 } \left( {d_u^{2} + d_v^{2}} \right) ^{1/2}
		\qquad & \text{if } \, \a < 0 ,
	\end{aligned}
	$$
	and the equality in the second, third or fifth bounds are attained for each $a$ if and only if $d_u=d_v$.
	The desired result is obtained by adding up for each $uv\in E(G)$.	
\end{proof}

The following result appears in [4].
\begin{lemma}\label{l:a_ir}
	If $a_{i} > 0$ for $1 \leq i \leq k$ and $r \in \mathbb{R}$, then
	$$
	\begin{aligned}
		\sum_{i=1}^{k} a_{i}^{r} \geq k^{1-r}\left(\sum_{i=1}^{k} a_{i}\right)^{r}, &\quad \text { if }\, r \leq 0 \text { or } r \geq 1, \\
		\sum_{i=1}^{k} a_{i}^{r} \leq k^{1-r}\left(\sum_{i=1}^{k} a_{i}\right)^{r}, &\quad \text { if }\, 0 \leq r \leq 1 .
	\end{aligned}
	$$
\end{lemma}

\begin{proposition}
	If $G$ is a graph with $m$ edges, then
	$$
	\begin{aligned}
		KA_{\a,\b}(G)\ge m^{1-\b}(M_1^{\a+1}(G))^{\b}
		&\quad \text { if }\, \b \leq 0 \text { or } \b \geq 1, \\
		KA_{\a,\b}(G)\le m^{1-\b}(M_1^{\a+1}(G))^{\b}
		&\quad \text { if }\, 0\le \b \le 1 \,.
	\end{aligned}
	$$
\end{proposition}
\begin{proof}
	If we take $a_i=d_u^\a+d_v^{a}$ and $r=\b$, by Lemma~\ref{l:a_ir} we have
	
	$$
	\begin{aligned}
		\sum_{uv\in E(G)} \left( d_u^\a+d_v^a \right) ^{\b}
		\geq m^{1-\b}\left(\sum_{uv\in E(G)} \left( d_u^\a+d_v^a \right)\right)^{\b},
		&\quad \text { if } \b \leq 0 \text { or } \b \geq 1, \\
		\sum_{uv\in E(G)} \left( d_u^\a+d_v^a \right) ^{\b}
		\leq m^{1-\b}\left(\sum_{uv\in E(G)} \left( d_u^\a+d_v^a \right)\right)^{\b},
		&\quad \text { if } 0 \leq \b \leq 1 .
	\end{aligned}
	$$
\end{proof}

Given a graph $G$, let us define the mean Sombor matrix $m\mathcal{S\!M}_{\a}(G)$ with entries
\begin{equation}\label{meanSmat}
{a}_{uv}:=\left\lbrace \begin{array}{ll}
	\left( \frac{d_u^\a + d_v^\a}{2}\right) ^{1/\a} , & \text{if } uv\in E(G)\, , \\
	0, & \text{otherwise} \, .
\end{array}
\right.
\end{equation}
One can easily check the following result about the trace of the matrix $m\mathcal{S\!M}_{\a}(G)^2$:
\begin{equation}\label{e:trPM}
	\operatorname{tr} \left(m\mathcal{S\!M}_{\a}(G)^2 \right)
	=
	\sum_{uv\in E(G)}\left( \frac{d_u^\a + d_v^\a}{2}\right) ^{2/\a} \, .
\end{equation}

Denote by $\sigma^2$ the variance of the sequence of the terms
$\left\lbrace \left( \frac{d_u^\a + d_v^\a}{2}\right) ^{1/\a}\right\rbrace $ appearing in the definition of $mSO_\a(G)$.
\begin{proposition}
	Let $G$ be a graph, then
	$$
	mSO_a(G)=\sqrt{ \frac{m}{2} \, \operatorname{tr}\left(m\mathcal{S\!M}_{\a}(G)^2\right) - m^2\sigma^2}\, .
	$$
\end{proposition}

\begin{proof}
	By the definition of $\sigma^2$, we have
	$$
	\sigma^2
	=
	\frac{1}{m}  \sum_{uv\in E(G)}\left(\left( \frac{d_u^\a + d_v^\a}{2}\right) ^{1/\a} \right) ^2
	-
	\left( \frac{1}{m} \sum_{uv\in E(G)}\left( \frac{d_u^\a + d_v^\a}{2}\right) ^{1/\a} \right) ^2
	\, ,
	$$
	then using the expression~\eqref{e:trPM} we have
	$$
	\sigma^2= \frac{1}{2m} \,\operatorname{tr}\left(m\mathcal{S\!M}_{\a}(G)^2\right) - \frac{1}{m^2} \, mSO_\a(G)^2 ,
	$$
	and the result follows from this equality.
\end{proof}

\begin{theorem}
	Let $G$ be any graph, then
	$$
	mSO_2(G) \le M_1(G)-M_2^{1/2}(G) \, ,
	$$
	where $M_2^{1/2}$ is the variable second Zagreb index $M_2^{\alpha}$ at $\alpha=1/2$, and the equality is attained if and only if each connected component of $G$ is a regular graph.
\end{theorem}

\begin{proof}
	Let be $\d$, $\D$ the minimum and maximum degree of $G$, respectively. Let's analyze the behavior of the function
	$$
	f(x,y)= \left( x+y-\sqrt{xy} \, \right) ^2 - \frac{x^2+y^2}{2}\, ,
	$$
	for $\d \le x\le y \le \D$. We have
	$$
	\begin{aligned}
		\frac{\partial f}{\partial x}(x,y)
		&= 2(x+y-\sqrt{xy} \,)\left( 1-\frac{1}{2}\sqrt{\frac{y}{x}} \, \right) -x \\
		&= x - 3\sqrt{xy} + 3y -\frac{y\sqrt{y}}{\sqrt{x}} \\
		&= \frac{x\sqrt{x}- 3x\sqrt{y} + 3y\sqrt{x} - y\sqrt{y}}{\sqrt{x}} \\
		&= \frac{\left( \sqrt{x} - \sqrt{y} \, \right) ^3}{\sqrt{x}} \le 0 \, ,
	\end{aligned}
	$$
	so $f$ is a decreasing function for each $y$. Thus, we have $f(x,y)\ge f(y,y)=0$, so
	$$
	x+y-\sqrt{xy}\ge \sqrt{\frac{x^2+y^2}{2}}\, ,
	$$
	and the equality is attained if and only if $x=y$. Therefore for any $uv\in E(G)$,
	$$
	d_u+d_v - \sqrt{d_ud_v} \ge \sqrt{\frac{d_u^2+d_v^2}{2}}\,
	$$
	and the equality is attained if and only if $d_u=d_v$. The desired result is obtained by adding up for each $uv\in E(G)$.
\end{proof}

\section{Discussion and conclusions}

We have introduced a degree--based variable topological index inspired on the power mean (also known
as generalized mean and H\"older mean).
We named this new index as the mean Sombor index $mSO_\alpha(G)$, see Eq.~(\ref{MG}).
For given values of $\alpha$, the mean Sombor index is related to well-known topological indices, in particular with
several Sombor indices.

In addition, through a QSPR study, we showed that mean Sombor indices are suitable to model
acentric factor, boiling point, heat capacity at constant pressure, standard enthalpy of vaporization,
enthalpy of formation, heat of vaporization at 25$^{\circ}$C,
enthalpy of vaporization, and entropy of octane isomers; see Section~\ref{QSPR}.

We have also discussed some mathematical properties of mean Sombor indices as well as stated bounds
and new relationships with known topological indices; see Section~\ref{ineq}, where the mean Sombor matrix
was also introduced in Eq.~(\ref{meanSmat}).

Finally, we would like to remark that, in addition to all the known indices that the mean Sombor index
reproduces, we discover the indices
$$
mSO_{-\infty}(G) \equiv {mSO}_{\alpha\to-\infty}(G) = \sum_{uv \in E(G)} \min(d_u,d_v)
$$
and
$$
mSO_\infty(G) \equiv {mSO}_{\alpha\to\infty}(G) = \sum_{uv \in E(G)} \max(d_u,d_v) \, ;
$$
which, from the QSPR study of Section~\ref{QSPR}, were shown to be good predictors of
the standard enthalpy of vaporization, the enthalpy of vaporization, and the heat of vaporization at
25$^{\circ}$C of octane isomers.
It is fair to mention that several known topological indices include the min/max functions; among them
we can mention the min-max (and max-min) rodeg index, the min-max (and max-min) sdi index, and the min-max
(and max-min) deg index, introduced in Ref.~\cite{VG10}. However, to the best of our knowledge, the indices
$mSO_{\pm \infty}(G)$ have not been theoretically studied before (for an exception where the equivalent
Stolarsky--Puebla indices have been computationally applied to random networks see~\cite{MAAS22}).
Thus, we believe that a theoretical study of these two new indices is highly pertinent.

\section*{Acknowledgment}

J.A.M.-B. acknowledges financial support from CONACyT (Grant No.~A1-S-22706) and
BUAP (Grant No.~100405811-VIEP2021).
E.D.M. and J.M.R. were supported by a grant from Agencia Estatal de Investigaci\'on (PID2019-106433GBI00/
AEI/10.13039/501100011033), Spain. J.M.R. was supported by the Madrid Government (Comunidad de Madrid-Spain) under the Multiannual Agreement with UC3M in the line of Excellence of University Professors (EPUC3M23), and in the context of the V PRICIT (Regional Programme of Research and Technological Innovation).


\end{document}